\documentclass[12pt]{amsart} \usepackage{amsmath,centernot}
\usepackage{amssymb} \usepackage{amsthm}
\usepackage[hidelinks]{hyperref}

\newtheorem{theorem}{Theorem}
\newtheorem{lemma}[theorem]{Lemma}
\newtheorem{proposition}[theorem]{Proposition}
\newtheorem{corollary}[theorem]{Corollary}
\newtheorem{claim}{Claim}
\numberwithin{theorem}{section}
\numberwithin{equation}{section}

\theoremstyle{definition}

\newtheorem{question}[theorem]{Question}
\newtheorem{remark}[theorem]{Remark}

\title[Recurrence, rigidity, and popular differences]{Recurrence, rigidity, and popular differences}

\date{\today}

\author{John T. Griesmer}
\email{jtgriesmer@gmail.com}

\begin{document}

\begin{abstract} We construct a set of integers $S$ such that every translate of $S$ is a set of recurrence and a set of rigidity for a weak mixing measure preserving system.  Here ``set of rigidity" means that enumerating $S$ as $(s_{n})_{n\in \mathbb N}$ produces a rigidity sequence.  This construction generalizes or strengthens results of Katznelson, Saeki (on equidistribution and the Bohr topology), Forrest (on sets of recurrence and strong recurrence), and Fayad and Kanigowski (on rigidity sequences).  The construction also provides a density analogue of Julia Wolf's results on popular differences in finite abelian groups. \end{abstract}

\maketitle

\section{Motivation}

Given a set of integers $A$, we write $d^{*}(A)$ for the upper Banach density of $A$; see Sections \ref{secPrecise} and \ref{secUBD} for exposition.  The \emph{difference set} $A-A$ is defined as $\{a-b:a,b\in A\}$.  A \emph{Bohr neighborhood} in the integers is a nonempty set of the form $\{n: p(n)>0\}$, where $p(n) = \sum_{j = 0}^{m} c_{j}\sin(\lambda_{j}n)+d_{j}\cos(\lambda_{j}n)$ is a trigonometric polynomial ($c_{j}, d_{j}, \lambda_{j}\in \mathbb R$).   Ruzsa asked the following in Section 5.8 of \cite[Part II]{RuzsaBook}.
\begin{question}[]\label{questionBohrDifference}
  If $d^{*}(A)>0$, must $A-A$ contain a Bohr neighborhood?
\end{question}
Question \ref{questionBohrDifference} was also asked in \cite{BergelsonRuzsa}, and in a slightly different form in \cite{GriesmerIsr}.
While we do not answer Question \ref{questionBohrDifference} in the present article, we answer the closely related Questions \ref{questionPopularBohr} and \ref{questionPopularDifferences}, providing support for a negative answer to Question \ref{questionBohrDifference}.  Most techniques for analyzing the difference set $A-A$ of a set $A$ having positive upper Banach density actually provide information about the \emph{$c$-popular differences} $P_{c}(A):=\{n: d^{*}(A\cap (A-n))>c\}$, where $c>0$ is some fixed number.  Note that $P_{c}(A)\subseteq A-A$ for each $c\geq 0$. The following questions now present themselves.
\begin{question}\label{questionPopularBohr}
  If $d^{*}(A)>0$, is there necessarily a $c>0$ such that $P_{c}(A)$ contains a Bohr neighborhood of some $n_{0}\in \mathbb Z$?
\end{question}
\begin{question}\label{questionPopularDifferences}
  Given $A\subseteq \mathbb{Z}$ having $d^{*}(A)>0$, is there a $c>0$, an $n_{0}\in \mathbb{Z}$, and a set $B\subseteq \mathbb{Z}$ having $d^{*}(B)>0$ such that $P_{c}(A)$ contains $n_{0}+B-B$?
\end{question}
 Our Corollary \ref{corPopularDifferences} answers Questions \ref{questionPopularBohr} and \ref{questionPopularDifferences} in the negative.  Corollary \ref{corPopularDifferences} follows from our main result, Theorem \ref{thmMainFull}, an ergodic theoretic construction generalizing some classical results in harmonic analysis and recent results in ergodic theory.

The next section summarizes some concepts from measure preserving dynamics and states our main results.

\section{Recurrence and rigidity}\label{secPrecise}

A \emph{measure preserving system} (or just ``system") $(X,\mu,T)$ is a probability measure space $(X,\mu)$, together with a map $T:X\to X$ preserving $\mu$: $\mu(T^{-1}D) = \mu(D)$ for every measurable $D\subseteq X$.  In this article, $T$ will always be invertible.   We do not mention the $\sigma$-algebra of measurable sets on $X$, as we only ever consider measurable subsets of $X$.   See any of \cite{CFS, EinsiedlerWard, Furstenberg, Glasner, Petersen, Walters} for standard definitions regarding measure preserving systems, including \emph{ergodicity}, \emph{weak mixing}, and the definition of isomorphism for measure preserving systems.

\subsection{Definitions}   A set $S\subseteq \mathbb Z$ is a:

 \begin{enumerate}
   \item[$\bullet$]  \emph{set of recurrence} if for every system $(X,\mu,T)$ and every $D\subseteq X$ having $\mu(D)>0$, there exists $n\in S$ such that $\mu(D\cap T^{n} D)>0$.
   
   \smallskip
   
   \item[$\bullet$] \emph{set of strong recurrence} if for every system $(X,\mu,T)$ and every $D\subseteq X$ with $\mu(D)>0$, there is a $c>0$ such that $\mu(D\cap T^{n} D)>c$ for infinitely many $ n \in S$.
  
     \smallskip
  
  \item[$\bullet$] \emph{set of optimal recurrence}\footnote{Also known as ``nice recurrence.''} if for every system $(X,\mu,T)$, all $\varepsilon>0$, and every $D\subseteq X$ having $\mu(D)>0$, there exists $ n\in \mathbb Z$ such that $\mu(D\cap T^{ n}D) > \mu(D)^{2} - \varepsilon$.
  
     \smallskip
      
   \item[$\bullet$]  \emph{set of rigidity} if $S$ is infinite and there is a nontrivial ergodic system $\mathbf X = (X,\mu,T)$ such that for all measurable $D\subseteq X$ and all $\varepsilon>0$, the set $\{ n \in S: \mu(D \triangle T^{ n}D)>\varepsilon\}$ is finite.  For such $\mathbf X$, we say that $S$ is \emph{set of rigidity for $\mathbf X$}.  If $S$ is a set of rigidity for a nontrivial weak mixing system $\mathbf X$, we say $S$ is a \emph{WM set of rigidity}.
 \end{enumerate}

Forrest \cite{Forrest} constructed an example of a set of recurrence which is not a set of strong recurrence.  See \cite{McCutcheon} for another proof of Forrest's result.

For technical reasons, some variations on the above definitions are in order.  Let $\mathbf{X}=(X,\mu,T)$ be a measure preserving system.

A set $S\subseteq \mathbb{Z}$ is a \emph{set of strong recurrence for} $\mathbf X$ if for every set $D\subseteq X$ having $\mu(D)>0$, there is a $c>0$ such that $\{n\in S: \mu(D\cap T^{n}D)>c\}$ is infinite.  The set $S$ is a \emph{WM set of strong recurrence} if it is a set of strong recurrence for every weak mixing system.

A sequence of integers $(s_{n})_{n\in \mathbb{N}}$ is a \emph{rigidity sequence for} $\mathbf X$ if for every set $D\subseteq X$, $\lim_{n\to \infty} \mu(D \triangle T^{s_{n}}D)=0$.  The set $S$ is a set of rigidity for $\mathbf X$ iff $(s_{n})_{n\in \mathbb{N}}$ is a rigidity sequence for $\mathbf X$ whenever $(s_{n})_{n\in \mathbb{N}}$ is an enumeration of $S$.

A system $\mathbf X$ is a \emph{rigid-recurrence system} if there is a set of rigidity $S$ for $\mathbf X$ such that every translate of $S$ is a set of recurrence, meaning $S+n$ is a set of recurrence for every $n\in \mathbb Z$.

\subsection{Results} Our main result is Theorem \ref{thmMainFull}; its proof uses measure concentration results similar to the use of Kleitman's Theorem in \cite{Forrest} and \cite{McCutcheon}, and to the measure concentration arguments of \cite{Wolf}.  See Sections \ref{secFourierTransforms} and \ref{secBohr} for  further discussion of Theorem \ref{thmMainFull} and its   relation to other work.

\begin{theorem}\label{thmMainFull} There is a set $S\subseteq \mathbb Z$ such that every translate of $S$ is both a set of recurrence and a WM set of rigidity.
\end{theorem}
We prove Theorem \ref{thmMainFull} in Section \ref{secFourierTransforms}.  An immediate consequence of the theorem (and its proof):

\begin{proposition}\label{propRigidRecurrenceExists}
  There is a weak mixing rigid-recurrence system.
\end{proposition}
Our first corollary generalizes the main result of \cite{Forrest}.
\begin{corollary}\label{corNeverStrong}
  There is a set $S\subseteq \mathbb Z$ such that every translate of $S$ is a set of recurrence, and no translate of $S$ is a WM set of strong recurrence.  Consequently, no translate of $S$ is a set of strong recurrence.
\end{corollary}

Corollary \ref{corNeverStrong} is proved in Section \ref{secProofsOfCorollaries}.

\subsection{Popular differences}  If $G$ is an abelian group, $t\in G$, and $A,B\subseteq G$, we write $A+t$ for $\{a+t : a\in A\}$, and $A-B$ for $\{a-b : a \in A,b\in B\}$.  The difference set $A-B$ is exactly the set $\{g \in G: B\cap (A-g) \neq \varnothing\}$.

If $G$ is a finite abelian group, $A\subseteq G$, and $c\geq 0$,  the \emph{$c$-popular difference set} $P_{c}(A)$ is the set of elements $g\in G$ such that there are more than $c|G|$ pairs $(a,b) \in A\times A$ satisfying $a-b = g$.  In other words, $P_{c}(A) := \{g: |A\cap (A-g)| > c|G|\}$.  Wolf constructed dense sets $A\subseteq \mathbb Z/N\mathbb Z$ where the popular difference set of $A$ does not contain any difference set $B-B$, where $B$ is not small.  To be more precise:
\begin{theorem}[\cite{Wolf}, Theorem 1.1]\label{thmWolf}
  There is a function $c:\mathbb N \to \mathbb R$ having $\lim_{n\to \infty} c(n)=0$, and there are sets $A_{N}\subseteq \mathbb Z/N \mathbb Z$ having $|A_{N}| > N/3$ such that $P_{c(N)}(A_{N})$ does not contain any difference set $B-B$, where $B\subseteq \mathbb{Z}/N\mathbb{Z} $ has $|B| > c(N)\cdot N$.
\end{theorem}

  Corollary \ref{corPopularDifferences} provides an analogue of Theorem \ref{thmWolf} in terms of upper Banach density, while Corollary \ref{corPopularRecurrence} is an analogue of Theorem \ref{thmWolf} for measure preserving systems.

\subsection{Upper Banach density and F{\o}lner sequences}    If $A\subseteq \mathbb{Z}$, the \emph{upper Banach density} of $A$ is $d^{*}(A):=\lim_{N\to \infty} \sup_{M \in \mathbb{Z}} \frac{|A\cap [M,M+N-1]|}{N}$.

A \emph{F{\o}lner sequence} for $\mathbb Z$ is a sequence $(\Phi_{j})_{j\in \mathbb N}$ of subsets of $\mathbb Z$ satisfying $\lim_{j\to \infty} \frac{|\Phi_{j} \triangle (\Phi_{j}+n)|}{|\Phi_{j}|}=0$ for every $n\in \mathbb Z$.  A sequence of intervals $I_{j}= \{M_{j}, M_{j}+1,\dots,N_{j}\}$ is a F{\o}lner sequence iff $\lim_{j\to \infty} |I_{j}|  = \infty$.  If $A\subseteq \mathbb Z$, and $\mathbf \Phi=(\Phi_{j})_{j\in \mathbb N}$ is a F{\o}lner sequence, we write $d_{\mathbf{\Phi}}(A)$ for $\lim_{j\to \infty} \frac{|A \cap \Phi_{j}|}{|\Phi_{j}|}$. Observe that $d^{*}(A)$ is the maximum of all values of $d_{\mathbf \Phi}(A)$ where $\mathbf \Phi$ is a F{\o}lner sequence.

For a set of integers $A\subseteq \mathbb Z$, an analogue of the $c$-popular difference set is $P_{c}(A):= \{n \in \mathbb Z: d^{*}(A\cap (A-n)) > c\}$.  For a measure preserving system $(X,\mu,T)$ and a set $D\subseteq X$, the analogue of the $c$-popular difference set is the \emph{$c$-popular recurrence set} $R_{c}(D):=\{n: \mu(D\cap T^{n}D)>c\}$.  From Corollary \ref{corNeverStrong} we will deduce the following analogues of Theorem \ref{thmWolf}.

\begin{corollary}\label{corPopularRecurrence}
  There is a weak mixing system $(X,\mu,T)$ such that for all $D\subseteq X$ having $\mu(D)>0$ and $D\cap TD = \varnothing$, and all $c>0$, the set $R_{c}(D)$ does not contain a set of the form $n_{0}+B-B$, where $d^{*}(B)>0$ and $n_{0}\in \mathbb Z$.  Consequently, $R_{c}(D)$ does not contain a Bohr neighborhood.
\end{corollary}
See Section  \ref{secBohr} for discussion of the Bohr topology.

  Corollary \ref{corPopularRecurrence} should be contrasted with the classical result that for every system $\mathbf X = (X,\mu,T)$, every $D\subseteq X$, and all $c<\mu(D)^{2}$, $R_{c}(D)$ contains a set of the form $U\setminus C$, where $U\subseteq \mathbb Z$ is a Bohr neighborhood of $0$ and $d^{*}(C)=0$. Furthermore, if $\mathbf{X}$ is weak mixing, the complement $E:=\mathbb Z \setminus R_{c}(D)$ has $d^{*}(E)=0$.  These results can be understood in terms of Fourier transforms of measures; see Section 1.3 of \cite{BHK} for a brief exposition.

\begin{corollary}\label{corPopularDifferences}
  For all $\varepsilon>0$, there is a set $A\subseteq \mathbb Z$ having $d^{*}(A)>\frac{1}{2} - \varepsilon$ such that for all $c>0$, the set $P_{c}(A)$ does not contain a set of the form $n_{0}+B-B$, where $d^{*}(B)>0$ and $n_{0}\in \mathbb Z$.  Consequently, $P_{c}(A)$ does not contain a Bohr neighborhood.
\end{corollary}

In fact, the examples $A$ in Corollary \ref{corPopularDifferences} can be taken to have \emph{uniform density}, meaning $\lim_{n\to \infty} \frac{|A\cap \Phi_{n}|}{|\Phi_{n}|}=d^{*}(A)$ for every F{\o}lner sequence $(\Phi_{n})_{n\in \mathbb N}$; this is proved in Proposition \ref{propUniform}.  In particular, there are syndetic\footnote{A set $S\subseteq \mathbb{Z}$ is \emph{syndetic} if there is a finite set $F$ such that $F+S = \mathbb{Z}$.} sets $A$ none of whose popular difference sets $P_{c}(A)$, $c>0$, contain a translate of a difference set $n_{0}+B-B$, where $d^{*}(B)>0$.

\begin{remark}
  We state the conclusion of Corollary \ref{corPopularRecurrence} in terms of differences sets $B-B$, where $d^{*}(B)>0$, rather than recurrence sets $R_{0}(D)$.  Lemma \ref{lemForwardAndReverseCorrespondence} shows that it makes no difference whether we use $R_{0}(D)$ or $B-B$ in the conclusion.
\end{remark}

%

\subsection{Outline of the article} Theorem \ref{thmMainFull} is proved in Section \ref{secFourierTransforms}.  We prove Corollaries \ref{corNeverStrong}, \ref{corPopularRecurrence}, and \ref{corPopularDifferences} in Section \ref{secProofsOfCorollaries}.  Section \ref{secBohr} discusses the Bohr topology and how our results are related to the existing literature. Section \ref{secMC} summarizes some applications of measure concentration underlying the proof of Theorem \ref{thmMainFull}.  Section \ref{secProofOfMain} contains the proof of Proposition \ref{propSimpleCase}, which is the main technical tool in the proof of Theorem \ref{thmMainFull}. Sections \ref{subsecSetsOfRecurrence}--\ref{secCorrespondence} contain lemmas needed in the preceding sections.

\section{Fourier transforms of measures, Proof of Theorem \ref{thmMainFull}}\label{secFourierTransforms}

Proposition \ref{propSimpleCase}, the main technical fact underlying the proof of Theorem \ref{thmMainFull}, is of independent interest.  Before stating the proposition, we briefly review some background on Fourier transforms of measures and Kronecker sets.

Here $\mathbb T:=\mathbb R/\mathbb Z$ is the torus with the usual topology.  We identify $\mathbb T$ with the interval $[0,1)$, so the functions $e_{n}:[0,1)\to \mathbb C$ given by $e_{n}(x):= \exp(2\pi i n x)$, $n\in \mathbb Z$, correspond to the characters of the group $\mathbb T$.

A Borel measure $\mu$ on a topological space $X$ is \emph{continuous} if $\mu(\{x\})=0$ for every $x\in X$.

If $\sigma$ is a measure on $\mathbb T$, its Fourier transform is the function $\hat{\sigma}:\mathbb Z\to \mathbb{C}$ given by $\hat{\sigma}(n) = \int e_{n}(\theta)\, d\sigma(\theta)$.  For a Borel probability measure $\sigma$ on $\mathbb{T}$ and a sequence of integers $(s_{n})_{n\in \mathbb N}$, the condition $\lim_{n\to \infty} \hat{\sigma}(s_{n})=1$ is equivalent to $\lim_{n\to \infty} \int |\exp(2\pi i s_{n} \theta) -1| \, d\sigma(\theta)=0$.  Under the additional assumption that $\sigma$ is continuous, these conditions imply that the sequence $(s_{n})_{n\in \mathbb{N}}$ is a rigidity sequence for a weak mixing system, as the following lemma states.

\begin{lemma}\label{lemRigidityCriteria}
   \begin{enumerate}
   \item[(i)] A sequence  $(s_{n})_{n\in \mathbb{N}}$ is a rigidity sequence for a weak mixing system if and only if there is a continuous measure $\sigma$ on $\mathbb T$ such that $\lim_{n\to \infty} \|e_{s_{n}}-1\|_{L^{1}(\sigma)}=0$.  Consequently:

\item[(ii)] A translated sequence $(s_{n}-m)_{n\in \mathbb N}$ is a rigidity sequence for a weak mixing system if and only if there is a continuous probability measure $\sigma$ on $\mathbb T$ such that $\lim_{n\to \infty} \|e_{s_{n}}-e_{m}\|_{L^{1}(\sigma)}=0$.
\end{enumerate}
\end{lemma}
See Proposition 2.10 of \cite{BergEtAl} for a proof of Lemma \ref{lemRigidityCriteria} (i) and further exposition.  Part (ii) follows from part (i) and multiplying the integrand in $\|e_{s_{n}-m}-1\|_{L^{1}(\sigma)}$ by $e_{m}$.

\subsection{Kronecker sets} The unit circle $\{z\in \mathbb C: |z| =1\}$ is denoted by $\mathcal S^{1}$.

A set $K\subseteq\mathbb T$ is a \emph{Kronecker set} if for every continuous function $f:K\to \mathcal S^{1}$ and all $\varepsilon>0$, there is an exponential function $e_{n}(x)=\exp(2\pi i n x)$ such that $\sup_{x\in K}|f(x)-e_{n}(x)| < \varepsilon$.  The group $\mathbb T$ contains a nonempty compact Kronecker set with no isolated points (\cite[Section 2]{HewittKakutani}; see also \cite{GrahamMcGehee,Rudin}).  In other words, $\mathbb T$ contains a nonempty perfect Kronecker set. Consequently, there are continuous probability measures supported on Kronecker sets. 

 \begin{proposition}\label{propSimpleCase}
Let $\sigma$ be a continuous probability measure supported on a nonempty perfect Kronecker set.

\begin{enumerate}
  \item[(a)] For all $\varepsilon>0$, every translate of the set
\[
  S_{\varepsilon} := \Bigl\{n : \int |e_{n}(x) - 1| \,d\sigma(x) <  \varepsilon\Bigr\}
\]
is a set of recurrence.

  \item[(b)] Let $f:K\to \mathcal{S}^{1}$ be a measurable function and $\varepsilon>0$.  Every translate of the set
\[
Q_{\varepsilon,f} := \Bigl\{n : \int |e_{n}(x) - f(x)| \,d\sigma(x) <  \varepsilon\Bigr\}
\]
is a set of recurrence.
\end{enumerate}\end{proposition}
Proposition \ref{propSimpleCase} is proved in Section \ref{secProofOfMain}.

 Proposition \ref{propSimpleCase} strengthens Theorem 2.2 of \cite{Katznelson}, which says\footnote{Theorem 2.2 of \cite{Katznelson} has a slightly weaker statement than that given here, but the proof is easily adapted to prove the stronger statement.} that there is a continuous probability measure $\sigma$ on $\mathbb T$ such that for each $\varepsilon>0$, $\{n: |\hat{\sigma}(n)|> 1-\varepsilon\}$ is dense in the Bohr topology of $\mathbb{Z}$.  The proposition also improves a special case of the main result of \cite{Saeki}, which says (in the special case) that if $\sigma$ is a continuous probability measure supported on a Kronecker set and $U\subseteq \mathbb C$ is an open subset of the closed unit disk, then $\{n: \hat{\sigma}(n) \in U\}$ is dense in the Bohr topology.   Proposition \ref{prop2Topologies} in the sequel explains why Proposition \ref{propSimpleCase} implies the earlier results. It is currently unknown whether every set dense in the Bohr topology is also a set of recurrence -- this is equivalent to Question \ref{questionBohrDifference}.  If so, Proposition \ref{propSimpleCase} merely recovers Theorem 2.2 of \cite{Katznelson} and the aforementioned special case of \cite{Saeki}.

\subsection{Recurrence is witnessed by finite approximations}  We will prove Theorem \ref{thmMainFull} by piecing together finite subsets of the sets $Q_{f,\varepsilon}$ in Proposition \ref{propSimpleCase}; for this we need Lemma \ref{lemRecurrenceIsSelective}.

 We say that $S\subseteq \mathbb Z$ is a \emph{set of $\delta$-recurrence} if for every measure preserving system $(X,\mu,T)$ and every $D\subseteq X$ such that $\mu(D)>\delta$, there exists $n\in S$ such that $\mu(D\cap T^{n}D)>0$.

\begin{lemma}[\cite{ForrestThesis}, Theorem 2.1]\label{lemRecurrenceIsFinite}
  A set $S\subseteq \mathbb Z$ is a set of recurrence if and only if for all $\delta>0$, there is a finite subset $S_{\delta}$ of $S$ such that $S_{\delta}$ is a set of $\delta$-recurrence.
\end{lemma}

\begin{lemma}\label{lemRecurrenceIsSelective}
  Let $S_{1} \supseteq S_{2} \supseteq S_{3} \supseteq \dots$ be a descending chain of subsets of $\mathbb Z$ such that every translate of $S_{j}$ is a set of recurrence for each $j$.  Then there is a set $S$ such that $S\setminus S_{j}$ is finite for each $j$ and every translate of $S$ is a set of recurrence.
\end{lemma}

\begin{proof}
For each $j\in \mathbb N$, let $S_{j}'\subseteq S_{j}$ be a finite set such that $S_{j}'+ m$ is a set of $\frac{1}{j}$-recurrence for every $|m|\leq j$; such finite sets exist by Lemma \ref{lemRecurrenceIsFinite}.    Then $S=\bigcup_{j \in \mathbb N} S_{j}'$ is the desired set.
\end{proof}

\begin{proof}[Proof of Theorem \ref{thmMainFull}]  Let $K\subseteq \mathbb T$ be a nonempty perfect Kronecker set, and let $(K_{m})_{m\in \mathbb Z}$ be a sequence of mutually disjoint nonempty perfect subsets of $K$.  Observe that each $K_{m}$ is a Kronecker set, for every subset of a Kronecker set is itself a Kronecker set.  For each $m$, let $\sigma_{m}$ be a continuous probability measure supported on $K_{m}$, and define $\sigma := \frac{1}{3}\sum_{m\in \mathbb Z}  2^{-|m|}\sigma_{j}$, so that $\sigma$ is a continuous probability measure on $K$. Let $f:K\to \mathcal S^{1}$ satisfy $f|_{K_{m}} = e_{m}$ for each $m$, and for each $j\in \mathbb N$ let
\[
  S_{j} =  \Bigl\{n\in \mathbb Z: \int |e_{n}-f|  \,d\sigma < \frac{1}{j}\Bigr\}.
\]
By Proposition \ref{propSimpleCase} (b), every translate of $S_{j}$ is a set of recurrence.   By Lemma \ref{lemRecurrenceIsSelective}, there is a set $S$ such that every translate of $S$ is a set of recurrence, and $S\setminus S_{j}$ is finite for every $n$.  Enumerating the elements of $S$ as $(s_{n})_{n\in \mathbb N}$, we have $e_{s_{n}} \to f$ in $L^{1}(\sigma)$, and it follows that $e_{s_{n}} \to f$ in $L^{1}(\sigma_{m})$ for every $m$.  Consequently, $e_{s_{n}} \to e_{m}$ in $L^{1}(\sigma_{m})$, by the definition of $f$, and we apply Lemma \ref{lemRigidityCriteria} to see that $S-m$ is a set of rigidity for every $m$. \end{proof}

\section{Proofs of Corollaries}\label{secProofsOfCorollaries}
Corollaries \ref{corNeverStrong} and \ref{corPopularRecurrence} will be proved with the aid of Lemma \ref{lemRigidNotStrong}.  Our approach to Corollary \ref{corPopularDifferences} requires some additional machinery in the form of the Jewett-Krieger theorem, which we review after the proofs of Corollaries \ref{corNeverStrong} and \ref{corPopularRecurrence}.

\begin{lemma}\label{lemRigidNotStrong}
  If $\mathbf X= (X,\mu,T)$ is a nontrivial weak mixing measure preserving system, $S$ is a set of rigidity for $\mathbf X$, and $m\in \mathbb Z\setminus \{0\}$, then $S+m$ is not a set of strong recurrence for $\mathbf X$.
\end{lemma}

\begin{proof}
  Let $S$ be a set of rigidity for $\mathbf X$. Let $D\subseteq X$ have $\mu(D)>0$ and $D\cap T^{m}D = \varnothing$.  Such a $D$ exists since $\mathbf X$ is weak mixing, and therefore totally ergodic.  Enumerate the elements of $S$ as $\{s_{n}:n\in \mathbb N\}$, so that $\mu(E\triangle T^{s_{n}}E)\to 0$ as $n\to \infty$ for every measurable set $E\subseteq X$.  Then $\mu(T^{m}D \triangle T^{s_{n}+m}D)\to 0$, so $\mu(D\cap T^{s_{n}+m}D)\to \mu(D\cap T^{m}D)=0$ as $n\to \infty$.  This shows that $S+m$ is not a set of strong recurrence for $\mathbf X$.
\end{proof}

\begin{proof}[Proof of Corollary \ref{corNeverStrong}] Lemma \ref{lemRigidNotStrong} and Theorem \ref{thmMainFull} together imply Corollary \ref{corNeverStrong}. \end{proof}

\begin{proposition}\label{propRigidRecurrencePopular}
  If $\mathbf X$ is a nontrivial rigid-recurrence system, then for every set $D\subseteq X$ such that $D\cap TD = \varnothing$ and every $c>0$, $R_{c}(D)$ does not contain a set of the form $n_{0}+B-B$, where $d^{*}(B)>0$.
\end{proposition}

\begin{proof} Let $S\subseteq \mathbb Z$ be a set of rigidity for $\mathbf X$ such that every translate of $S$ is a set of recurrence. Assume, to get a contradiction, that there is a set $D\subseteq X$ with $\mu(D)>0$ such that $D\cap TD= \varnothing$,  and for some $c>0$, the set $R_{c}(D)=\{n: \mu(D\cap T^{n}D)>c\}$ contains a set of the form $n_{0}+B-B$, where $d^{*}(B)>0$.

 Every translate of $S$ is a set of recurrence, so every translate of $S+1$ is a set of recurrence. Lemma \ref{lemShiftInfinite} implies $(S+1)\cap (n_{0}+B-B)$ is infinite, and consequently $(S+1)\cap R_{c}(D)$ is infinite.  On the other hand, if we enumerate $S$ as $(s_{n})_{n\in \mathbb N}$ we have $\lim_{n\to \infty} \mu(D\cap T^{s_{n}+1}D)=0$ by an argument similar to the proof of Lemma \ref{lemRigidNotStrong}.  The last limit implies the set $(S+1)\cap R_{c}(D)$ is finite, contradicting the earlier conclusion that it is infinite. \end{proof}

\begin{proof}[Proof of Corollary \ref{corPopularRecurrence}]  Theorem \ref{thmMainFull} implies the existence of nontrivial weak mixing rigid-recurrence systems, so
take $\mathbf{X}$ to be any weak mixing rigid-recurrence system and apply Proposition \ref{propRigidRecurrencePopular}.    \end{proof}

\subsection{Topological systems and unique ergodicity}  A \emph{topological system} $(X,T)$ is a compact metric space $X$ together with a self homeomorphism $T:X\to X$.  We say that $(X,T)$ is \emph{uniquely ergodic} if there is only one $T$-invariant Borel probability measure on $X$.  See any of \cite{CFS, EinsiedlerWard, Furstenberg, Glasner, Petersen, Walters} for background on topological systems.

In the proof of Corollary \ref{corPopularDifferences}, we will need minimal uniquely ergodic models of measure preserving systems. Such models are provided by the Jewett-Krieger theorem \cite{Jewett,Krieger}, which says that if $(X,\mu,T)$ is an ergodic measure preserving system, then there is a minimal uniquely ergodic topological system $(\tilde{X},\tilde{T})$ with invariant probability measure $\tilde{\mu}$ such that $(X,\mu,T)$ is isomorphic to $(\tilde{X},\tilde{\mu},\tilde{T})$.  We can take the space $\tilde{X}$ to be totally disconnected.

Below we list some well-known properties of uniquely ergodic topological systems.  Let $(X,T)$ be a uniquely ergodic topological system with invariant probability measure $\mu$ on a totally disconnected space $X$.  Assume $\mu$ is continuous, so that $\mu(\{x\})=0$ for all $x\in X$.

\begin{enumerate}
  \item[(U1)] If $K\subseteq X$ is a clopen set, then for every F{\o}lner sequence $(\Phi_{j})_{j\in \mathbb N}$, we have $\lim_{j\to \infty} \frac{1}{|\Phi_{j}|}\sum_{n\in \Phi_{j}} 1_{K}(T^{n}x) = \mu(K)$, for every $x\in X$.

      \item[(U2)]  If $K\subseteq X$, $n\in \mathbb Z$, and $x\in X$, let $A:=\{m: T^m x\in K\}$. Then $A\cap (A-n)= \{m: T^{m}x\in K\cap T^{-n}K\}$.  Consequently, if $K$ is clopen, then for every F{\o}lner sequence $\mathbf \Phi$, $d_{\mathbf \Phi}(A\cap (A-n)) = \mu(K\cap T^{-n}K)$.

      \item[(U3)]  For all $\varepsilon>0$, there is a clopen set $D\subseteq X$ having $\mu(D)>\frac{1}{2}-\varepsilon$ such that $D\cap TD = \varnothing$.
 \end{enumerate}
To prove statement (U3), take $N$ to be sufficiently large and construct a Rohlin tower $\{C, TC,\dots, T^{2N-1}C\}$ for $(X,\mu,T)$ where the $T^{i}C$ are mutually disjoint clopen sets such that $\mu(X \setminus \bigcup_{i=0}^{2N-1}T^{i} C) > 1-\varepsilon$.  Let $D=\bigcup_{i=0}^{N-1}T^{2i}C$. Then $\mu(C)\geq \frac{1-\varepsilon}{2N}$,  $\mu(D) \geq N\cdot \mu(C) \geq N \frac{1-\varepsilon}{2N}  = \frac{1-\varepsilon}{2}$, and $D\cap TD = \varnothing$.   Such a Rohlin tower may be constructed by following the proof of \cite[Lemma 2.45]{EinsiedlerWard}, starting with a clopen set $A$ in the proof.

We say a set $A\subseteq \mathbb Z$ has \emph{uniform density} $\alpha$ if $d_{\mathbf{\Phi}}(A)=\alpha$ for every F{\o}lner sequence $\mathbf{\Phi}$, and we write $d_{U}(A)=\alpha$.

The following proposition immediately implies Corollary \ref{corPopularDifferences}.
\begin{proposition}\label{propUniform}
  For all $\varepsilon>0$, there is a set $A\subseteq \mathbb Z$ having uniform density $\geq \frac{1}{2}-\varepsilon$, such that for all $c>0$, the set $P_{c}(A)$ does not contain a translate of a difference set $n_{0}+B-B$, where $d^{*}(B)>0$.
\end{proposition}

\begin{proof}
  Let $\mathbf X = (X,\mu,T)$ be a weak mixing system as in the conclusion of Corollary \ref{corPopularRecurrence}.  By the Jewett-Krieger theorem, there is a minimal uniquely ergodic topological dynamical system $\tilde{\mathbf{X}}=(\tilde{X},\tilde{T})$ with invariant probability measure $\tilde{\mu}$ such that $(\tilde{X},\tilde{\mu},\tilde{T})$ is isomorphic to $(X,\mu,T)$, and $\tilde{X}$ is totally disconnected.

Fix $\varepsilon>0$,  and  apply Property (U3) to find a clopen $D\subseteq \tilde{X}$ such that $D\cap \tilde{T}D = \varnothing$ and $\mu(D)>\frac{1}{2}-\varepsilon$.  Let $x\in \tilde{X}$ and $A = \{m:\tilde{T}^{m}x \in D\}$.  Then property (U1) implies $d_{U}(A) = \mu(D)>\frac{1}{2}-\varepsilon$. Property (U2) implies $\mu(D\cap \tilde{T}^{-n}D) = d_{U}(A\cap (A-n))=d^{*}(A\cap (A-n))$ for all $n\in \mathbb N$.  The conclusion now follows from Corollary \ref{corPopularRecurrence}.\end{proof}

\section{The Bohr topology and recurrence}\label{secBohr}
The \emph{Bohr topology} on $\mathbb Z$ is the weakest topology on $\mathbb Z$ such that every homomorphism $\chi:\mathbb Z\to \mathbb T$ is continuous.  See Chapter 5 of \cite[Part II]{RuzsaBook} for more on the Bohr topology and difference sets.

\begin{remark}\label{remSubbase}
  A subbase for the Bohr topology consists of sets of the form $\mathcal B(\alpha,U):=\{n \in \mathbb Z: n\alpha \in U\}$, where $U\subseteq \mathbb T$ is open and $\alpha \in \mathbb T$.  In other words, the Bohr topology consists of arbitrary unions of finite intersections of sets of the form $\mathcal B(\alpha,U)$.
\end{remark}

\begin{proposition}\label{prop2Topologies}
  If every translate of $S\subseteq \mathbb Z$ is a set of recurrence, then $S$ is dense in the Bohr topology on $\mathbb Z$.
\end{proposition}

\begin{proof}  This is a standard fact, so we only outline the proof.  The following standard facts imply the conclusion:
\begin{enumerate}
  \item[(i)] Every  neighborhood in the Bohr topology has positive upper Banach density.

\item[(ii)]  Every Bohr neighborhood $U$ of the identity $0_{\mathbb Z}$ contains a difference set $V-V$ of a Bohr neighborhood $V$ of the identity.

\item[(iii)]  Every Bohr neighborhood of a point contains a translate of a Bohr neighborhood of $0_{\mathbb{Z}}$.
\end{enumerate}
\vspace{-.2in}

\end{proof}

Whether the converse of Proposition \ref{prop2Topologies} holds is an open problem, equivalent to Question \ref{questionBohrDifference}: if the answer to Question \ref{questionBohrDifference} is ``yes," then every set $S$ dense in the Bohr topology of $\mathbb{Z}$ is a set of recurrence, hence every translate of such a set is a set of recurrence.

\subsection{Rigidity and the Bohr topology}\label{secRigidity}
Answering Question 3.5 of \cite{BergEtAl}, Fayad and Kanigowski prove the following.

 \begin{theorem}[\cite{FaKa}, Theorem 2]\label{thmFaKa} There is a set of rigidity $S \subseteq \mathbb Z$ such that for all irrational $\alpha$, $\{n\alpha \textup{ mod 1}: n\in S\}$ is dense in $\mathbb R/\mathbb Z$, and for all rational $\alpha$, $\{n \alpha \textup{ mod 1}: n \in S\} = \{n \alpha \textup{ mod 1}: n \in \mathbb Z\}$.\end{theorem}

In light of Proposition \ref{prop2Topologies} and Remark \ref{remSubbase}, Theorem \ref{thmMainFull} implies Theorem \ref{thmFaKa} and the following strengthening.

\begin{theorem}\label{thmBohrDenseForZ2}
  There is a set of rigidity $S \subseteq \mathbb Z$ such that $S$ is dense in the Bohr topology of $\mathbb Z$.
\end{theorem}

\subsection{Equidistribution}\label{secEqui}

A set $S\subseteq \mathbb Z$ is \emph{equidistributed} if there is a sequence of finite subsets $S_{j} \subseteq S$ such that for every $\alpha \in (0,1)$,
\begin{align}\label{eqEquidist}
  \lim_{j\to \infty} \frac{1}{|S_{j}|} \sum_{ n\in S_{j}}  \exp(2\pi i n \alpha) = 0.
\end{align}
It follows easily from the definition that $S$ is equidistributed iff every translate of $S$ is equidistributed.  Furthermore, standard arguments show that if $(S_{j})_{j\in \mathbb N}$ is a sequence of finite sets satisfying equation (\ref{eqEquidist}), then for every measure preserving system $(X,\mu,T)$ and every $D\subseteq X$, we have
\[
\lim_{j\to \infty} \frac{1}{|S_{j}|}\sum_{n\in S_{j}} \mu(D\cap T^{n}D) \geq \mu(D)^{2}.
\]
  Consequently, every translate of every equidistributed set is a set of optimal recurrence.

\subsection{A hierarchy of recurrence properties} Consider the following properties of a set of integers $S$.

\begin{enumerate}
  \item[(R1)] $S$ is equidistributed.

\smallskip

\item[(R2)] Every translate of $S$ is a set of optimal recurrence.

\smallskip

\item[(R3)] Every translate of $S$ is a set of strong recurrence.

\smallskip

\item[(R4)] Every translate of $S$ is a set of recurrence.

\smallskip

\item[(R5)] $S$ is dense in the Bohr topology of $\mathbb Z$.
 \end{enumerate}
 We have (R1) $\implies$ (R2)  $\implies$ (R3) $\implies$ (R4) $\implies$ (R5); the first implication is discussed above, and the last implication is Proposition \ref{prop2Topologies}.  Katznelson's and Saeki's examples (Theorem 2.2 of \cite{Katznelson} and the main theorem of \cite{Saeki}) show that (R5) does not imply (R1).

 Corollary \ref{corNeverStrong} shows that (R4) does not imply (R3), and  in fact (R4) does not imply that any translate of $S$ is a set of strong recurrence.  Corollary \ref{corNeverStrong} provides the first explicit proof that (R4) $\centernot\implies$ (R1), although it is plausible that the constructions of \cite{Forrest} and \cite{McCutcheon} produce sets satisfying (R4) and not (R3).

Does (R5) imply (R4)?  This is Question \ref{questionBohrDifference}. Does (R3) imply (R2)?  Does (R2) imply (R1)?  The answers are likely ``no,'' but we are unaware of a proof.

\section{Measure concentration}\label{secMC}  The main result of this section is Lemma \ref{lemKleitmanRecurrence}, which is the key lemma in proving Proposition \ref{propSimpleCase}.  Our approach is virtually identical to the proof of \cite[Lemma 2.4]{Wolf}, and we require some results from \cite{McDiarmid}. To state those, we introduce some terminology.

Let $k,r\in \mathbb N$, let $\Lambda_k:=\{\lambda \in \mathbb C: \lambda ^{k} = 1\}$, the group of $k^{\text{th}}$ roots of unity, where the group operation is ordinary multiplication. Consider the group $\Lambda_{k}^{r}:=(\Lambda_{k})^{r}$, the $r^{\text{th}}$ cartesian power of $\Lambda_{k}$.  We write the group operation on $\Lambda_{k}^{r}$ using multiplicative notation.  If $A, B\subseteq \Lambda_{k}^{r}$, we write $A\cdot B$ for the product set $\{ab: a\in A, b\in B\}$, and $A^{-1}A$ for the difference set\footnote{Perhaps we should write ``quotient set", but we abuse terminology instead.} $\{a^{-1}b : a,b\in A\}$.

Write an element of $\Lambda_{k}^{r}$ as $x = (x_{1},\dots,x_{r})$. Let $d_0$ denote one half of euclidean distance, so that $d(x,y):=\sum_{j=1}^{r} d_0(x_j,y_j)$ defines a translation invariant metric on $\Lambda_{k}^{r}$ with diameter $\leq r$.  For $t \geq 0$ and $x\in \Lambda_{k}^{r}$, let $U_{t}(x)$ denote the $d$-ball of radius $t$ around the element $x = (x_{1},\dots,x_{r})$, meaning  $U_{t}(x) = \{y \in \Lambda_{k}^{r}: \sum_{j=1}^{r}d_0(x_{j},y_{j}) \leq t\}$.  We refer to $U_{t}(x)$ as the \emph{Hamming ball} of radius $t$ around $x$.

The following lemma may be proved in exactly the same manner\footnote{We use the parameter $r$ where \cite{McDiarmid} uses the parameter $n$.  Take $ \gamma=1$ in the conclusion of Proposition 7.7 of \cite{McDiarmid} to get the expression $-t^{2}/2r$ in the exponent here.} as Proposition 7.7 of \cite{McDiarmid}.
\begin{lemma}\label{lemMeasureConcentration}
  Let $A\subseteq \Lambda_{k}^{r}$ be nonempty, $t\geq 0$ and let $\alpha = |A|/|\Lambda_{k}^{r}|$.  Then
  \[
    |A\cdot U_{t}(0)|/|\Lambda_{k}^{r}| \geq 1 - \alpha^{-1}\exp(-t^{2}/2r).
  \]
  Consequently, for every $x\in \Lambda_{k}^{r}$,
  \[
    |A\cdot U_{t}(x)|/|\Lambda_{k}^{r}| \geq 1 - \alpha^{-1}\exp(-t^{2}/2r).
  \]
\end{lemma}

\begin{lemma}\label{lemKleitmanGeneral}
  Let $A\subseteq \Lambda_{k}^{r}$, $x\in \Lambda_{k}^{r}$, and $t>0$. If $(A^{-1}A)\cap U_{t}(x) = \varnothing$, then $|A|/|\Lambda_{k}^{r}| \leq \exp(-t^{2}/4r)$.
\end{lemma}
\begin{proof}
  Assume $A^{-1}A\cap U_{t}(x)=\varnothing$.  Write $\alpha$ for $|A|/|\Lambda_{k}^{r}|$. Note that
  \[
  A^{-1}A\cap U_{t}(x) = \varnothing \iff A\cap (A\cdot U_{t}(x))=\varnothing,
  \] which implies $|A| + |A\cdot U_{t}(x)| \leq |\Lambda_{k}^{r}|$, or
  \begin{align}\label{eqDensities}
     \frac{|A|}{|\Lambda_{k}^{r}|} + \frac{|A\cdot U_{t}(x)|}{|\Lambda_{k}^{r}|} \leq 1.
  \end{align}
   Lemma \ref{lemMeasureConcentration} implies $|A\cdot U_{t}(x)|/|\Lambda_{k}^{r}| \geq 1 - \alpha^{-1} \exp(-t^{2}/2r)$; substituting into inequality (\ref{eqDensities}) we get
  \[
  \frac{|A|}{|\Lambda_{k}^{r}|} + 1-\alpha^{-1}\exp(-t^{2}/2r)\leq 1,
  \] meaning $\alpha + 1 - \alpha^{-1}\exp(-t^{2}/2r) \leq 1$.  Solving for $\alpha$, we find $\alpha^{2} \leq \exp(-t^{2}/2r)$, or $\alpha \leq \exp(-t^{2}/4r)$, which is the desired conclusion.
\end{proof}

The next lemma is a qualitative restatement of Lemma \ref{lemKleitmanGeneral}.

\begin{lemma}\label{lemKleitmanRecurrence}
For all $\varepsilon,\,\delta>0$ and all $k\in \mathbb N$,  there exists $N=N(\delta,\varepsilon,k)$ such that: if $r>N$ and $A\subseteq \Lambda_{k}^{r}$ has $|A| \geq \delta |\Lambda_{k}^{r}|$, then for all $x\in \Lambda_{k}^{r}$, there are $a_{1},a_{2}\in A$ such that $a_{1}^{-1}a_{2}\in U_{r\cdot \varepsilon}(x)$.
\end{lemma}

\begin{proof}
If $(A^{-1}A) \cap U_{r\cdot \varepsilon}(x) =\varnothing$,  Lemma \ref{lemKleitmanGeneral} implies
\[
|A|/|\Lambda_{k}^{r}| \leq \exp(-\varepsilon^{2}r/4).
\]  For $r$ sufficiently large, this implies $|A|/|\Lambda_{k}^{r}|<\delta$.
\end{proof}

\begin{remark}\label{remIsometry}  Consider $\Lambda_{k}^{r}$ as the set of functions $f:\{1,\dots, r\}\to \Lambda_{k}$. If $m$ is normalized counting measure on $\{1,\dots,r\}$ and $x\in \Lambda_{k}^{r}$, then $\{y\in \Lambda_{k}^{r}: \|y-x\|_{L^{1}(m)}<\varepsilon\}$ is the set of $y$ satisfying $\frac{1}{r}\sum_{j=1}^{r} |y_{j}-x_{j}| < \varepsilon$, or
\[  \sum_{j=1}^{r} \tfrac{1}{2} |y_{j} - x_{j}| <  r\cdot \varepsilon/2.\]
Hence, the $L^{1}(m)$-ball of radius $\varepsilon$ around $x$ is the Hamming ball $U_{r\cdot \varepsilon/2}(x)$.
\end{remark}

\section{Proof of Proposition \ref{propSimpleCase}.}\label{secProofOfMain}

Throughout this section, $K\subseteq \mathbb T$ will be a nonempty perfect Kronecker set and $\sigma$ will be a continuous probability measure supported on $K$.  In the proof of Proposition \ref{propSimpleCase}, we repeatedly use the following property of probability measures supported on Kronecker sets:

\begin{center}
  \begin{minipage}{0.9\textwidth}
     For all measurable functions $f:K\to \mathcal S^{1}$ and all $\varepsilon>0$, there is a character $e_{n}(x)= \exp(2\pi i n x)$ such that $\|f-e_{n}\|_{L^{1}(\sigma)}<\varepsilon$.
  \end{minipage}
\end{center}

\begin{proof}[Proof of Proposition \ref{propSimpleCase}]   Our main task is proving part (a).   Fix $\varepsilon>0$.   Recall that $S_{\varepsilon}:=\{n:\int |e_{n}-1| \,d\sigma < \varepsilon\}$, and we aim to prove that every translate of $S_{\varepsilon}$ is a set of recurrence. We write a translate $S_{\varepsilon}+m$ as
\begin{align*}
  S_{\varepsilon}+m &= \Bigl\{n+m: \int |e_{n}-1| \,d\sigma<\varepsilon\Bigr\}\\
  &= \Bigl\{n: \int |e_{n-m}-1| \,d\sigma<\varepsilon\Bigr\} \\
  &= \Bigl\{n: \int |e_{n}-e_{m}| \,d\sigma<\varepsilon\Bigr\};
\end{align*}
multiply the integrand $|e_{n-m}-1|$ by $|e_{m}|$ in the second line to get the third line.  For the remainder of the proof, we fix a set of integers $B$ having $d^{*}(B)>0$, we fix an $m\in \mathbb Z$, and we aim to prove that $(B-B)\cap (S_{\varepsilon}+m) \neq \varnothing$.  In other words: \begin{claim}\label{claimA1} There are $a,\,b\in B$ such that $\int |e_{a-b}(x) -  e_{m}(x)|\, d\sigma(x)<\varepsilon$. \end{claim}
Since $B$ is an arbitrary set having $d^{*}(B)>0$,  Proposition \ref{propSimpleCase} follows from Claim \ref{claimA1} (by way of Lemma \ref{lemForwardAndReverseCorrespondence}).

We will derive Claim \ref{claimA1} from Lemma \ref{lemKleitmanRecurrence}.  We continue to denote by $\Lambda_{k}$ the multiplicative group of $k^{\text{th}}$ roots of unity.  We first choose $k$ to be sufficiently large that there is a function $g:K\to \Lambda_{k}$ such that
  \begin{equation}\label{eqApproxByg}
    \|e_{m}-g\|_{L^{1}(\sigma)}<\varepsilon/3.
  \end{equation} We choose $r>N(d^{*}(B),\varepsilon/6,k)$ in Lemma \ref{lemKleitmanRecurrence}. We also choose $r$ large enough that there is a partition $\mathcal P = (P_{1},\dots,P_{r})$ of $K$ into $r$ sets of equal $\sigma$-measure such that there is a $g:K\to \Lambda_{k}$ satisfying inequality (\ref{eqApproxByg}) which is constant on elements of $\mathcal P$.  For the remainder of the proof, we fix such $g$ and $\mathcal P$.

To prove Claim \ref{claimA1} it is enough to approximate $g$ by $e_{a-b}$ where $a,b\in B$: we will show that there are $a$, $b\in B$ such that $\|e_{a-b} - g\|_{L^{1}(\sigma)}<2\varepsilon/3$.  Rewriting $e_{a-b}$ as $e_{a}\overline{e_{b}}$, we will prove:

\begin{claim}\label{claimA2}
  There are $a, b\in B$ such that $\|e_{a}\overline{e_{b}} - g\|_{L^{1}(\sigma)}<2\varepsilon/3$.
\end{claim}

Now consider the set $H_{r,k}$ of functions $\psi: K \to \Lambda_{k}$ which are $\mathcal P$-measurable, with $\mathcal P$ as defined above.    Note that $H_{r,k}$ is a group under pointwise multiplication, isomorphic to $\Lambda_{k}^{r}$.  Furthermore, the $L^{1}(\sigma)$ ball of radius $\varepsilon$ around $g$ may be identified with a Hamming ball $U_{r\cdot \varepsilon/2}(x)$ in $\Lambda_{k}^{r}$, as mentioned in Remark \ref{remIsometry}.  Since $K$ is a Kronecker set, functions $\psi\in H_{r,k}$ can be approximated in $L^{1}(\sigma)$ by functions of the form $e_{n}$.

\begin{claim}\label{claimPsiByShifts}

There are $\psi_{1},\psi_{2} \in H_{r,k}$ satisfying simultaneously

\textup{(i)} $\|\psi_{1}\overline{\psi_{2}}-g\|_{L^{1}(\sigma)}<\varepsilon/3$, and

\textup{(ii)} there are $n_{1}, n_{2}\in B$ and $z\in \mathbb Z$ such that $\|e_{n_{j}+z}-\psi_{j}\|_{L^{1}(\sigma)}<\varepsilon/3$, $j=1$, $2$.

\end{claim}

Inequalities (i) and (ii), together with the triangle inequality,  imply $\|e_{n_{1}+z}\overline{e_{n_{2}+z}} - g\|_{L^{1}(\sigma)}<2\varepsilon/3$. The last inequality implies Claim \ref{claimA2},  since $e_{n_{1}+z}\overline{e_{n_{2}+z}}=e_{n_{1}}\overline{e_{n_{2}}}$.  The remainder of the proof of Proposition \ref{propSimpleCase} is dedicated to proving Claim \ref{claimPsiByShifts}.

Inequality (i) involves only elements of $H_{r,k}$.  Regarding $\psi_{1},\psi_{2}$, and $g$ as elements of $\Lambda_{k}^{r}$, the inequality (i) says that $\psi_{1}\overline{\psi_{2}}$ lies in the Hamming ball $U_{r\cdot \varepsilon/6}(g)$ centered at $g$.  Using this correspondence, we will derive the following claim from Lemma \ref{lemKleitmanRecurrence}.

\begin{claim}\label{claimMC}
  Let  $N=N(\delta,\varepsilon/6,k)$ be as in Lemma \ref{lemKleitmanRecurrence}.   If $r>N$ and $A\subseteq H_{r,k}$ has $|A| \geq \delta |H_{r,k}|$, then for all $g\in H_{r,k}$, there are $a_{1},a_{2}\in A$ such that $\|a_{1}a_{2}^{-1}-g\|_{L^{1}(\sigma)}< \varepsilon/3$.
\end{claim}
\begin{proof}[Proof of Claim \ref{claimMC}]
  The group $H_{r,k}$ of $\Lambda_{k}$-valued $\mathcal P$-measurable functions is isomorphic to the group $\Lambda_{k}^{r}$, and the natural isomorphism induces an isometry between $H_{r,k}$ with the $L^{1}(\sigma)$ metric and the metric space $(\Lambda_{k}^{r},d')$, where $d'=2d$, and $d$ is the metric defined on $\Lambda_{k}^{r}$ in Section \ref{secMC}.  Now Claim 4 follows immediately from Lemma \ref{lemKleitmanRecurrence} and Remark \ref{remIsometry}.
\end{proof}
Claim \ref{claimMC} says that we can prove Claim \ref{claimPsiByShifts} if a substantial portion of the $\psi \in H_{r,k}$ can be $\varepsilon/3$-approximated in $L^{1}(\sigma)$ by exponentials $e_{n}$, where $n\in B$. In fact, it is enough to show that there is a single translate $B+z$ of $B$ such that a substantial portion of $H_{r,k}$ can be $\varepsilon/3$-approximated by exponentials of the form $e_{n}$, where $n$ is in $B+z$, since $e_{a+z}\overline{e_{b+z}} = e_{a}\overline{e_{b}}$.

\begin{claim}\label{claimUBDforSimple}
  For $z\in \mathbb Z$, let $\tilde{B}_{z} \subseteq H_{r,k}$ be the set of $\psi\in H_{r,k}$ satisfying
   \begin{align}\label{eqDenseTranslate}
     \text{There exists $n\in B$ such that $\|e_{n+z}-\psi\|_{L^{1}(\sigma)}<\varepsilon/3$.}
   \end{align}
  Then $|\tilde{B}_{z}| \geq d^{*}(B)|H_{r,k}|$ for some $z$.
\end{claim}

\begin{proof}
  For each $\psi \in H_{r,k}$, let $n(\psi)\in \mathbb Z$ be such that $\|e_{n(\psi)}-\psi\|_{L^{1}(\sigma)}<\varepsilon/3$.  Such an $n(\psi)$ exists, since $K$ is a Kronecker set.  Furthermore, choose $n(\psi)$ so that $n(\psi)\neq n(\psi')$ if $\psi\neq \psi'$.  Let $E_{r,k} = \{n(\psi) : \psi\in H_{r,k}\}$, so that $|E_{r,k}| = |H_{r,k}|$, and $E_{r,k}\subseteq \mathbb Z$.  Apply Lemma \ref{lemUBDchar} to find a translate $B+z$ of $B$ such that $|(B+z)\cap E_{r,k}| \geq d^{*}(B) |E_{r,k}|$.  The definition of $n(\psi)$ implies $|\tilde{B}_{z}| \geq |(B+z)\cap E_{r,k}|$, so we have the desired conclusion.
\end{proof}

Since we chose $r> N(d^{*}(B),\varepsilon/6, k)$, Claims \ref{claimMC} and \ref{claimUBDforSimple} together imply Claim \ref{claimPsiByShifts}, which implies Claim \ref{claimA2}.  Since Claim \ref{claimA2} implies Claim \ref{claimA1}, this completes the proof of Proposition \ref{propSimpleCase}, Part (a).

Part (b) follows immediately from Part (a): since $\sigma$ is supported on a Kronecker set, there is an exponential $e_{m}$ such that $\|f-e_{m}\|_{L^{1}(\sigma)}<\varepsilon/2$, so $Q_{\varepsilon,f}$ contains $S':=\{n: \|e_{n}-e_{m}\|_{L^{1}(\sigma)}<\varepsilon/2\}$.   Observe that $S'-m = S_{\varepsilon/2}$ (as defined in Part (a) of the Proposition), so every translate of $S'$ is a set of recurrence.  Consequently, every translate of $Q_{\varepsilon,f}$ is a set of recurrence. \end{proof}

 \section{Sets of recurrence}\label{subsecSetsOfRecurrence}

The lemmas in this section use the following property of ergodic systems, which is actually equivalent to ergodicity.

\begin{enumerate}
  \item[(E)]  If $(X,\mu,T)$ is an ergodic system and $C,D\subseteq X$ have positive measure, then there is an $n\in \mathbb Z$ such that $\mu(C\cap T^{n}D)>0$. \end{enumerate}

\begin{lemma}\label{lemExpanding}
  For a set of integers $S$, the following conditions are equivalent.
  \begin{enumerate}
    \item[(i)]   Every translate of $S$ is a set of recurrence.
    \item[(ii)]  For every measure preserving system $(X,\mu,T)$ and every pair of sets $D_{1},D_{2}\subseteq X$: if there is an $n\in \mathbb Z$ such that $\mu(D_{1}\cap T^{n}D_{2})>0$, then there is an $m\in S$ such that $\mu(D_{1}\cap T^{m}D_{2})>0$.
    \item[(iii)] If $(X,\mu,T)$ is an ergodic measure preserving system and $\mu(D)>0$, then $\mu\bigl(\bigcup_{n\in S} T^{n}D\bigr)=1$.
  \end{enumerate}
\end{lemma}

\begin{proof}
(i)$\implies$(ii): Suppose every translate of $S$ is a set of recurrence, $(X,\mu,T)$ is a measure preserving system, $n\in \mathbb Z$, $D_{1},D_{2}\subseteq X$, and $\mu(D_{1}\cap T^{n}D_{2})>0$.   Since $S-n$ is a set of recurrence, there is an $m\in S$ such that
  \[
    \mu(D_{1}\cap T^{n}D_{2} \cap T^{m-n}(D_{1}\cap T^{n}D_{2}))>0
  \]
  which implies $\mu(D_{1}\cap T^{m}D_{2})>0$, since $T^{m-n}(D_{1}\cap T^{n}D_{2})= T^{m-n}D_{1}\cap T^{m}D_{2}$.

(ii)$\implies$(iii):  Let $\mathbf X = (X,\mu,T)$ be an ergodic system, and let $S\subseteq \mathbb Z$ satisfy (ii).  For $D\subseteq X$, write $S\cdot D$ for $\bigcup_{n\in S} T^{n}D$.  Suppose, to get a contradiction, that $\mu(D)>0$ and $\mu(S\cdot D)<1$.  Let $E=X\setminus (S\cdot D)$.  Then $\mu(E)>0$ and the ergodicity of $\mathbf X$ implies that there is an $n$ such that $\mu(E \cap T^{n}D)>0$.  Property (ii) implies that there is an $m\in S$ such that $\mu(E \cap T^{m}D)>0$, and this is the desired contradiction.

(iii)$\implies$(i):  Suppose $S$ satisfies (iii).  It suffices to show that $S$ is a set of recurrence, since condition (iii) is translation invariant.  Suppose $B\subseteq \mathbb Z$ satisfies $d^{*}(B)>0$; we must show that $B\cap (B-n) \neq \varnothing$ for some $n\in S$.

Apply Lemma \ref{lemCorrespondence} to find an ergodic system $(X,\mu,T)$ and a set $D\subseteq X$ such that $\mu(D)\geq d^{*}(B)$ and $\mu(D\cap T^{n}D) \leq d^{*}(B\cap (B-n))$ for every $n\in \mathbb Z$.  It then suffices to find $n\in S$ such that $\mu(D\cap T^{n}D)>0$.  Condition (iii) implies $\mu(S\cdot D)=1$, so $\mu(D\cap (S\cdot D)) = \mu(D)$, and there must be an $n\in S$ such that $\mu(D\cap T^{n}D)>0$.  \end{proof}

\begin{lemma}\label{lemShiftInfinite}
  If every translate of $S$ is a set of recurrence, $d^{*}(A)>0$, and $n\in \mathbb{Z}$, then $S\cap (n+A-A)$ is infinite.
\end{lemma}

This lemma may seem obvious. To see why we expend effort proving it, observe that there are finite sets of recurrence: the singleton $\{0\}$ is a set of recurrence.

In the proof we need the following property of weak mixing systems: if $\mathbf{X}$ is ergodic and $\mathbf{Y}$ is weak mixing, then the product system $\mathbf{X}\times \mathbf Y$ is ergodic.
\begin{proof}[Proof of Lemma \ref{lemShiftInfinite}]
  It suffices to show that $S\cap (A-A)$ is infinite whenever $d^{*}(A)>0$, since the hypothesis on $S$ is translation invariant.  Fix such a set $A$.

By Lemma \ref{lemCorrespondence}, $A-A$ contains a set of the form $R_{0}(D):=\{n\in \mathbb{Z}: \mu(D\cap T^{n}D)>0\}$, where $(X,\mu,T)$ is an ergodic measure preserving system and $D\subseteq X$ has $\mu(D)\geq d^{*}(A)$.  So it suffices to show that $S\cap R_{0}(D)$ is infinite.

  Let $F=\{m_{1},\dots,m_{r}\}\subseteq \mathbb Z$ be a finite set; we will show there is an  $m \in S \setminus F$ such that $\mu(D\cap T^{m}D)>0$.  Let $n_{0}\in \mathbb Z \setminus F$.  Let $\mathbf Y = (Y,\nu,R)$ be a nontrivial weak mixing system, and let $E\subseteq Y$ have positive measure and satisfy $E \cap R^{m-n_{0}}E=\varnothing$ for all $m\in F$; this is possible because $\mathbf Y$ is totally ergodic\footnote{meaning the system $(Y,\nu,R^{n})$ is ergodic for each $n\neq 0$, a consequence of weak mixing.} and $m-n_{0} \neq 0$ for $m \in F$.

  Consider the product system $\mathbf Z = (X\times Y, \mu\times \nu,T\times R) = \mathbf X \times \mathbf Y$, and let $C_{1}, C_{2} \subseteq Z$ be the following:
\[
  C_{1} := D \times E,\ \ \ C_{2}:= D\times R^{-n_{0}}E.
 \]
 Write $\eta$ for $\mu\times \nu$ and $\tilde{T}$ for $T\times R$.

Then for all $m\in F$, we have $C_{1}\cap \tilde{T}^{m} C_{2}= \varnothing$, since $E\cap R^{m-n_{0}}E = \varnothing$.  Since $\mathbf Y$ is weak mixing and $\mathbf X$ is ergodic, the system $\mathbf Z$ is ergodic.  We conclude that there is an $n$ such that $\eta(C_{1}\cap \tilde{T}^{n}C_{2})>0$.    Every translate of $S$ is a set of recurrence, so Lemma \ref{lemExpanding} implies $\eta(C_{1}\cap \tilde{T}^{m}C_{2})>0$ for some $m\in S$.  Since $C_{1}\cap \tilde{T}^{m}C_{2}=\varnothing$ for all $m\in F$,  there is an $m\in S\setminus F$ such that $\eta(C_{1}\cap \tilde{T}^{m}C_{2})>0$, and $\mu(D\cap T^{m}D)>0$ for such $m$.  Since $F$ was an arbitrary finite set, we have shown that $S\cap R_{0}(D)$ is infinite.  \end{proof}

\section{Characterization of upper Banach density}\label{secUBD}

\begin{lemma}\label{lemUBDchar}
 Let $F\subseteq \mathbb Z$ be finite and let $A\subseteq \mathbb Z$.  There is a $ n\in \mathbb Z$ such that $|(A-n)\cap F| \geq d^{*}(A)\cdot |F|$.
\end{lemma}

\begin{proof}  We outline the proof in this paragraph and present an explicit proof in the next.  First, find a long interval $I$ having $\frac{1}{|I|}|I\cap A| \approx d^{*}(A)$.  We may assume that $I=\{1,\dots, N\}$ for some large $N$, $F\subseteq I$, and $(\max F)/N$ is small.   Consider the intersections $I\cap (A-n)\cap F$.  Summing their cardinalities, we get $\sum_{n=1}^{N} |I\cap (A-n)\cap F| \approx |F|\cdot |I\cap A|$, as each element of $I\cap A$ is counted approximately $|F|$ times in the sum.  So on average, $|I\cap (A-n)\cap F|$ is approximately $\frac{1}{N}\cdot |F|\cdot |I\cap A| \approx |F|\cdot d^{*}(A)$, and there must be an $n$ such that $|(A-n)\cap F| \geq d^{*}(A)$.

Here is an explicit proof: Let $\mathbf \Phi = (\Phi_{j})_{j\in \mathbb N}$ be a F{\o}lner sequence such that $d^{*}(A) = d_{\mathbf{\Phi}}(A)$.  Let $\varepsilon, \varepsilon'>0$.  For sufficiently large $j$, we have
  \begin{align*}
   \frac{1}{|\Phi_{j}|} \sum_{ n \in \Phi_{j}}|(A-n)\cap F| &= \frac{1}{|\Phi_{j}|} \sum_{ n'\in F}\sum_{ n\in \Phi_{j}} 1_{A-n}(n')  \\
   &= \frac{1}{|\Phi_{j}|} \sum_{n'\in F} \sum_{n\in \Phi_{j}} 1_{A}(n+n') \\
    &=  \sum_{ n' \in F} \frac{1}{|\Phi_{j}|}|A\cap (\Phi_{j}- n')|\\
    &\geq   \sum_{n'\in F} \frac{1}{|\Phi_{j}|}(|A\cap \Phi_{j}| - \varepsilon) \\
    &= \frac{1}{|\Phi_{j}|} (|A\cap \Phi_{j}| -\varepsilon)|F|\\
    & \geq  (d^{*}(A)-\varepsilon)|F|.
  \end{align*}
  From the above inequalities, we see that for $\varepsilon$ sufficiently small, the average $\frac{1}{|\Phi_{j}|} \sum_{ n \in \Phi_{j}}|(A-n)\cap F|$ is at least $d^{*}(A)|F|-\varepsilon'$. The pigeonhole principle then implies $|(A-n)\cap F| \geq d^{*}(A)|F| - \varepsilon'$ for some $n$.  Using the integrality of $|(A-n)\cap F|$, we conclude that $|(A-n)\cap F| \geq d^{*}(A)|F|$ for some $ n$.  \end{proof}

 Deriving the following corollary from Lemma \ref{lemUBDchar} is an easy exercise.

 \begin{corollary}
    For $A\subset \mathbb Z$, $d^{*}(A)$ is the largest number $\delta$ satisfying the following condition: for all finite sets $F\subseteq \mathbb Z$, there exists $n$ such that $|(A-n)\cap F| \geq \delta |F|$.
 \end{corollary}

\section{The correspondence principle}\label{secCorrespondence}

The following version of Furstenberg's correspondence principle is a special case of Proposition 3.1 in \cite{BHK}.

\begin{lemma}\label{lemCorrespondence}
  Let $A\subseteq \mathbb Z$.  Then there is an ergodic measure preserving system $(X,\mu,T)$ and a set $D\subseteq X$ having $\mu(D)=d^{*}(A)$, and
  \[
    d^{*}(A\cap (A-m)) \geq \mu(D\cap T^{-m}D)
  \]
  for all $m\in \mathbb Z$.
\end{lemma}

The next lemma is a partial converse of the correspondence principle.  It is a special case of Proposition 3.2.7 in \cite{McCutcheonBook}.

\begin{lemma}\label{lemReverseCorrespondence}
  Let $(X,\mu,T)$ be a measure preserving system and let $D\subseteq X$ have $\mu(D)>0$.  Then there is a set $A\subseteq \mathbb Z$ having $d^{*}(A)\geq \mu(D)$ such that
  \[
    A - A \subseteq \{m: \mu(D\cap T^{-m}D)>0\}.
  \]
  \end{lemma}

  Combining Lemmas \ref{lemCorrespondence} and \ref{lemReverseCorrespondence} yields the following.

  \begin{lemma}\label{lemForwardAndReverseCorrespondence}
    Let $\delta>0$ and let $E\subseteq \mathbb Z$.  The following are equivalent:
   \begin{enumerate}
     \item[(i)] $E$ contains
   $A-A$, where $A\subseteq \mathbb Z$ and $d^{*}(A)\geq \delta$.
   \item[(ii)]$E$ contains $R_{0}(D)$ where $(X,\mu,T)$ is a measure preserving system and $D\subseteq X$ has $\mu(D)\geq \delta$.
   \item[(iii)]  $E$ contains $R_{0}(D)$ where $(X,\mu,T)$ is an ergodic measure preserving system and $D\subseteq X$ has $\mu(D)\geq \delta$.
 \end{enumerate}  \end{lemma}
Unfortunately, the connection between popular differences $P_{c}(A)$ and the popular recurrence sets $R_{c}(D)$ for $c>0$ is not as straightforward as the connection between difference sets $A-A$ and recurrence sets $R_{0}(D)$.  Rather than attempt to prove an analogue of Lemma \ref{lemForwardAndReverseCorrespondence} for popular differences and popular recurrence, in Section \ref{secProofsOfCorollaries} we resort to some ad hoc constructions using uniquely ergodic systems.

\bibliographystyle{amsplain}
\frenchspacing
\bibliography{For_Z_Bibliography}

\end{document}